\documentclass[11pt,letterpaper]{amsart}

\usepackage[pdftex,colorlinks,linkcolor=blue,citecolor=red,filecolor=black]{hyperref}

\usepackage{amsmath}
\usepackage{amsthm}
\usepackage{amssymb}
\usepackage{xfrac}
\usepackage{faktor}
\usepackage{mathtools}
\usepackage{mathrsfs}
\usepackage{tikz}
\usepackage{framed}
\usepackage{upgreek}
\usepackage{bbm}
\usepackage{bm}
\usepackage{graphicx}
\usepackage{cite}

\usepackage{enumitem}
\setlist[enumerate]{label=({\roman*})}

\newtheorem{theorem}{Theorem}
\newtheorem{proposition}[theorem]{Proposition}
\newtheorem{corollary}[theorem]{Corollary}
\newtheorem{lemma}[theorem]{Lemma}

\theoremstyle{definition}

\newtheorem{remark}[theorem]{Remark}
\newtheorem{remarks}[theorem]{Remarks}

\numberwithin{equation}{section}
\numberwithin{theorem}{section}

\def\<{\langle}
\def\>{\rangle}

\def\e{\epsilon}

\def\Z{\mathbb{Z}}
\def\R{\mathbb{R}}

\def\C{\mathbb{C}}

\def\vp{\varphi}
\def\L{L_\varphi}
\def\R{\mathbb R}
\def\Z{\mathbb Z}
\def\W{\mathcal W}
\def\PG{P_{\mathrm{G}}}
\def\Gab{G^{\mathrm{ab}}}
\def\oG{\overline{G}}
\begin{document}
\bibliographystyle{plain} \title[Gurevi\v{c} pressure and equidistribution]
{Gurevi\v{c} pressure and equidistribution for amenable extensions of countable state Markov shifts}

\author{Richard Sharp} 
\address{Mathematics Institute, University of Warwick,
Coventry CV4 7AL, U.K.}

\thanks{\copyright 2025. This work is licensed by a CC BY license.}


\begin{abstract}
We obtain a weighted equidistribution theorem for amenable skew product extensions of 
countable state Markov shifts satisfying the BIP property. We also show, without requiring the BIP property, that 
Gurevi\v{c} pressure for an amenable skew product agrees with the Gurevi\v{c} pressure for the abelianized system.
This had been proved by Dougall and Sharp in the case where the base is a subshift of finite type.
The equality of Gurevi\v{c} pressures is a key part of the proof of the equidistribution result.
\end{abstract}

\maketitle

%
%
%
%
%
%
%
%
%

\section{Introduction}

It is well-known that, for mixing subshifts of finite type
$\sigma : \Sigma \to \Sigma$  (and hyperbolic systems more generally), weighted averages of periodic points 
become equidistributed with respect to the Gibbs measure for the weighting potential, and this continues to hold 
in the countable state case provided the so-called \emph{big images and pre-images} (BIP) property 
is satisfied. 
Precisely, if $\delta_x$ is the Dirac measure at $x$ and 
$\tau_{x,n} = \frac{1}{n}\sum_{j=0}^{n-1} \delta_{\sigma^jx}$ then, for suitably regular $\vp : \Sigma \to \R$,
and writing
\[
\vp^n := \vp + \vp \circ \sigma + \cdots \vp \circ \sigma^{n-1},
\]
the sequence of probability measures
\[
\left(\sum_{\sigma^nx=x} e^{\vp^n(x)} \right)^{-1}
\sum_{\sigma^n x=x} e^{\vp^n(x)} \tau_{x,n}
\]
converges weakly to $\mu_\vp$, the Gibbs measure for $\vp$, as $n \to \infty$.
(For subshifts of finite type, see Bowen \cite{Bowen1971} when $\vp=0$ and Parry \cite{Parry1988},
Parry and Pollicott \cite{PP} for general H\"older continuous $\vp$. The result 
in \cite{Parry1988} and \cite{PP} is given for flows
but the shift case follows directly by considering a constant suspension.
For the countable state case with BIP, the equidistribution result can be obtained from the stronger large deviation results in
Takahasi \cite{Takahasi2019}.)
However, more interesting phenomena occur if we consider extensions of $\Sigma$ by a (finitely generated, infinite)
countable group $G$ and only consider those periodic points for $\sigma : \Sigma \to \Sigma$
which correspond to periodic points in the extension.
More precisely, given $\psi : \Sigma \to G$, define the skew product extension
$T_\psi : \Sigma \times G \to \Sigma \times G$ by
\[
T_\psi(x,g) = (\sigma x,g\psi(x)).
\]
Then we consider points $x \in \Sigma$ such that $T_\psi^n(x,e)=(x,e)$, i.e. $\sigma^nx=x$ and 
$\psi_n(x)=e$, where $e$ is the identity in $G$
and
\[
\psi_n(x):= \psi(x)\psi(\sigma x) \cdots \psi(\sigma^{n-1}x).
\] 
(Note
that $T_\psi^n(x,e)=(x,e)$ if and only if $T_\psi^n(x,g)=(x,g)$ for all $g \in G$.)
Given a (sufficiently nice)
function $\vp : \Sigma \to \R$, we might seek to understand the limiting behaviour of the measures
\[
\mathfrak M_n :=
\left(\sum_{\substack{\sigma^nx=x \\ \psi_n(x)=e}} e^{\vp^n(x)} \right)^{-1}
\sum_{\substack{\sigma^n x=x \\ \psi_n(x)=e}} e^{\vp^n(x)} \tau_{x,n}
\]
This problem was studied when $\Sigma$ is a subshift of finite type and $G = \Z^d$, in which case,
subject to a mixing condition on the extension,
the weak limit exists and is equal to the Gibbs measure for $\vp + \langle \xi,\psi \rangle$, for some 
(unique) $\xi = \xi(\vp) \in \R^d$ (where $\langle \cdot,\cdot \rangle$ is the standard inner product on $\R^d$).
This value $\xi$ is characterised by minimizing the pressure function $P(\vp+\langle w,\psi\rangle,\sigma)$
over $w \in \R^d$, or by the variational principle
\[
h(\mu_{\vp+\langle \xi,\psi \rangle}) = \sup\left\{h(\nu) + \int \vp \, d\nu \hbox{ : } \nu \in \mathcal M_\sigma(\Sigma),
\  \int \psi \, d\nu=0\right\},
\]
where $\mathcal M_\sigma(\Sigma)$ is the set of $\sigma$-invariant Borel probability measures on $\Sigma$
and $h(\nu)$ is the measure-theoretic entropy of $\nu$ with respect to $\sigma$.
In the technically more involved situation of Anosov flows and homology, the analogous
convergence statement for $\vp=0$ was proved by Sharp \cite{Sharp93} (see also Lalley \cite{LalleyAdv}
and Babillot and Ledrappier \cite{BabLed})
and for general H\"older continuous
$\vp$ by Coles and Sharp \cite{ColesSharp}, and it is easy to recover the shift case from the results in these
papers.
It is now apparent that a version of this result should hold whenever $G$ is amenable:
again a result for Anosov flows was proved by Dougall and Sharp \cite{DougallSharp2021}
and one can adapt this to subshifts of finite type (indeed, the analysis in \cite{DougallSharp2021}
proceeds via subshifts of finite type and symbolic dynamics to obtain the flow result).
In the amenable case, the limit measure is exactly the same as one would obtain if one abelianized the group $G$. More precisely, if $\overline{G}\cong \Z^d$ is the torsion-free part of the abelianization of $G$
and $\pi : G \to \overline{G}$ is the projection homomorphism, then
the limiting measure is $\mu_{\vp+\langle \xi,\bar \psi\rangle}$, where 
$\bar \psi : \Sigma \to \overline{G}$ is defined by $\bar \psi = \pi \circ \psi$ and
$\xi=\xi(\vp)$ is as above. (If $d=0$ then $\xi=0$ and one simply obtains $\mu_\vp$ in the limit.)

Our first main result is to extend this equidistribution for amenable skew products to mixing countable state Markov shifts 
satisfying BIP. Omitting some more technical hypotheses, we show that if
\begin{itemize}
\item
$\sigma : \Sigma \to \Sigma$ satisfies BIP;
\item
$G$ is amenable;
\item
$T_\psi : \Sigma \times G \to \Sigma \times G$ is topologically mixing; and
\item
$\vp : \Sigma \to \R$ is locally H\"older continuous
\end{itemize}
then there exists $\xi \in \R^d$ such that
$\mathfrak M_n$ converges to 
$\mu_{\vp + \langle \xi,\bar \psi \rangle}$, as $n \to \infty$,
with respect to the weak$^*$ topology on $\mathcal M_\sigma(\Sigma)$.
The result is stated precisely as Theorem \ref{th:equi_for_G}.
The key technical assumption is that, writing $|\cdot|$ for the $2$-norm on $\R^d$,
\begin{equation}\label{eq:cramer-intro}
\sup_{x \in \Sigma} \sum_{\sigma^n y = x} e^{\vp^n(y) + r|\bar \psi^n(y)|}
\end{equation}
is finite for some $r>0$ -- this is somewhat analogous to Cram\'er's condition for random walks
(existence of an exponential moment) or to the entropy gap assumption
that appears in \cite{BCKM} and \cite{KM}, for example.
Of course, for a subshift of finite type, (\ref{eq:cramer-intro}) is automatically finite for all $r>0$.

The equidistribution result is proved via large deviations, and to prove the latter, we need to show that,
when $G$ is amenable, the Gurevi\v{c} pressure 
of $\vp$ with respect to the skew product $T_\psi$ is equal to the Gurevi\v{c} pressure 
of $\vp$ with respect to the abelianized skew product $T_{\bar \psi}$ (which, in turn, is equal to
the Gurevi\v{c} pressure of $\vp + \langle \xi,\bar \psi \rangle$ with respect to $\sigma$).
The largest part of the paper is devoted to proving the first equality.
 To motivate this, we briefly outline
some older results for random walks, spectral geometry and group theory, before returning to 
dynamical questions.

A classical result of Kesten for symmetric random walks on countable groups is that the exponential rate of 
return to the identity is equal to $1$ if the group is amenable and is less than $1$ otherwise \cite{Kesten}.
Variants of this amenability dichotomy were subsequently obtained in the settings of 
the spectrum of the Laplacian for coverings of Riemannian manifolds
\cite{Brooks1981}, \cite{Brooks1985}, group theory and graph theory  
\cite{Cohen1982}, \cite{Grigorchuk1980}, \cite{Northshield1992}, \cite{Northshield2004}
and critical exponents for discrete groups \cite{CDS},
\cite{CDST}, \cite{Dougall2019}, \cite{DougallSharp2016}, \cite{Roblin}, \cite{Stadlbauer2013}.
The latter results also have a dynamical interpretation in terms of geodesic flows
(over manifolds with negative sectional curvatures) and, more generally, Anosov flows with a time-reversing 
involution.  In this setting, with time-reversal symmetry, one compares the growth rate of the number of periodic orbits 
of the flow (which is equal to the topological entropy) with the corresponding quantity for the lift of
the flow to a regular cover: the growth rates agree if and only if the covering group is amenable.
For general Anosov flows, without this additional symmetry, the amenability dichotomy fails (even for abelian covers).
However, in \cite{DougallSharp2021}, Dougall and the author showed that a modified result holds, where the comparison is between growth rates for the given cover and the maximal abelian subcover:
these two quantities agree if and only if the covering group is amenable.
The argument in \cite{DougallSharp2021} (with some errors rectified in the associated correction)
used symbolic dynamics, where the Anosov flow is 
coded by a suspension flow over a subshift of finite type. In particular, it boils down showing equality of Gurevi\v{c} 
pressure for a skew product extension associated to the covering group with the corresponding quantity for its abelianization.
Here we give a new proof, modelled on the argument given for random walks 
in \cite{DougallSharp2024}, which works in the countable state case and which gives a cleaner proof.
(Both proofs have their roots in ideas of Roblin \cite{Roblin} and Stadlbauer \cite{Stadlbauer2019}.)
We stress that, for the direction ``amenability implies equality''
stated in Theorem \ref{th:main} below, the shift $\sigma : \Sigma \to \Sigma$ does not need to satisfy BIP.
(The recent paper \cite{GomezTerhesiu2025} obtains the conclusion of Theorem \ref{th:main} when the base is a full branch Gibbs--Markov map satisfying BIP.)

We now state this result in detail (although some definitions will be deferred until 
section \ref{sec:CSMS}).
Given a locally  H\"older continuous function $\vp : \Sigma \to \R$, we have induced functions
$\tilde \vp : \Sigma \times G \to \R$ and $\bar \vp : \Sigma \times \overline G \to \R$, defined by
$\tilde \vp(x,\cdot) =\vp(x)$ and $\bar \vp(x,\cdot)=\vp(x)$, respectively.
To avoid over-cluttering formulae, we will slightly abuse notation by denoting all three functions
$\vp$, $\bar \vp$ and $\tilde \vp$ by $\vp$.
We have the following inequalities for the (Gurevi\v{c}) pressure of these functions:
\begin{equation}\label{eq:gur_pres_inequalities}
\PG(\vp,T_\psi) \le \PG(\vp,T_{\bar \psi}) \le \PG(\vp + \langle w,\bar \psi \rangle,\sigma),
\end{equation}
where the second inequality holds for all $w \in \R^d$. (This will be justified in section \ref{sec:skew_prod_ext}.)

\begin{theorem}\label{th:main}
Suppose that 
\begin{itemize}
\item
$\sigma : \Sigma \to \Sigma$ is a (one-sided) 
countable state Markov shift, 
\item
$G$ is a countable group,
\item 
$T_\psi : \Sigma \times G \to \Sigma \times G$ be a topologically transitive skew-product extension,
where $\psi((x_i)_{i=1}^\infty)=\psi(x_1)$, and
\item
$\vp : \Sigma \to \mathbb R$ is a locally H\"older continuous function.
\end{itemize}
If $G$ is amenable then
$\PG(\vp,T_\psi) = \PG(\vp,T_{\bar \psi})$.
\end{theorem}

\begin{remarks}\label{rem:remark to main}
(i) Since $T_\psi : \Sigma \times G \to \Sigma \times G$ is topologically transitive, $\sigma :\Sigma \to \Sigma$ is also topologically transitive, so we do not need to include this as a separate assumption on the base.

\noindent
(ii) If $\Sigma$ satisfies BIP
and $\vp$ and $\psi$ satisfy the same conditions needed for Theorem \ref{th:equi_for_G},
then the converse of statement also holds:
if $\PG(\vp,T_\psi) = \PG(\vp,T_{\bar \psi})$ then $G$ is amenable. This includes the case of subshifts of finite type (where the moment condition is automatic). The proof, which relies heavily on results
of Stadlbauer \cite{Stadlbauer2013} is given in section \ref{sec:eq imp am}.
\end{remarks}

We will now outline the contents of the rest of the paper. 
Section \ref{sec:CSMS} and 
section \ref{sec:skew_prod_ext} provide background on countable state Markov shifts, Gurevi\v{c} pressure and skew product extensions.
Section \ref{sec:am imp eq} contains the proof of Theorem \ref{th:main}.
Section \ref{sec:bip} introduces the BIP property and makes precise the assumptions we shall impose in our 
equidistribution results and section \ref{sec:equi_for_skew} contains the statement and proof of these results.
Finally, section \ref{sec:eq imp am} establishes the converse to Theorem \ref{th:main} subject to the assumption of BIP.

\section{Countable state Markov shifts}\label{sec:CSMS}
Let $S$ be a countable set and let $A$ be a matrix indexed by $S \times S$ with
entries in $\{0,1\}$. We call $S$ the set of states and $A$ the transition matrix. 
We write $A(s,s')$ for the $(s,s')$ entry of $A$.
We then let 
\[
\Sigma=\{x=(x_n)_{n=0}^\infty \hbox{ : } x_n \in S \mbox{ and } A(x_n,x_{n+1})=1 \ \forall n \in \Z^+\}.
\]
Let $\sigma : \Sigma \to \Sigma$ be the shift map defined by
$(\sigma x)_n=x_{n+1}$. Then we call $(\Sigma,\sigma)$
a (one-sided) countable state Markov shift. If $S$ is finite, then $(\Sigma,\sigma)$ is called a (one-sided) subshift of finite type.

A word $w=(w_1,\ldots,w_{n}) \in S^n$ is called \emph{allowed} if $A(w_i,w_{i+1})=1$ for all $i=1,\ldots,n-1$
and we call $[w]=\{x \in \Sigma \hbox{ : } x_i=w_i, \ \forall i=1,\ldots,n\}$ the cylinder defined by $w$.
Let 
$\mathcal W_n \subset S^n$ denote the set of all allowed words of length $n$ and let
$\mathcal W := \bigcup_{n=1}^\infty \mathcal W_n$. 
For $a \in S$, we write
\[
\mathcal W_n(a) := \{w=(w_1,\ldots,w_n)\in \mathcal W_n \hbox{ : } w_1=a \mbox{ and } 
A(w_n,a)=1\}.
\]
If $w$ is a finite work and $x$ is either a finite word or element of $\Sigma$ the we let $wx$ denote
the concatenation.

The topology on $\Sigma$ is the topology generated by cylinder sets. The shift map is continuous and
we havre that $\Sigma$ is compact if and only if $S$ is finite. 
We say that $\sigma : \Sigma \to \Sigma$ is \emph{topologically transitive} if there is a
point with a dense orbit. 
We say that
$\sigma : \Sigma \to \Sigma$ is \emph{topologically mixing}
if for all $a,b \in S$ there exists $N(a,b)$
such that for all $n \ge N(a,b)$ there exists $w \in \mathcal W_n$ with $awb \in \mathcal W$.
(This is equivalent to the usual definition of topological mixing in topological dynamical systems.)

We say that $\vp : \Sigma \to \R$ is \emph{locally H\"older continuous} if
\[
\sup_{w \in \W_n} \sup_{x,y \in [w]} |\vp(x)-\vp(y)| \le C\theta^n,
\]
for some $C>0$ and $0<\theta<1$, for all $n \ge 1$.
If we wishe to specify $\theta$ then we say that $\vp$ is
$\theta$-locally H\"older continuous.
We write
\[
\vp^n(x) := \vp(x) +\vp(\sigma x) + \cdots + \vp(\sigma^{n-1}x).
\]
A locally H\"older function satisfies the following estimates.

\begin{lemma}[Lemma 1 of \cite{Sarig-etds}]\label{lem:bounded_variation}
Let $\vp : \Sigma \to \mathbb R$ be locally H\"older continuous.
Then, for all $a \in S$ and all $o,o' \in [a]$ there exists $B>0$ such that
\[
|\vp^n(wo)-\vp^n(wo')| \le B
\]
for all $w \in \mathcal W_n$ such that $wa \in \mathcal W_{n+1}$.
\end{lemma}

Assume that $\sigma$ is topologically transitive and that
$\vp : \Sigma \to \R$ is locally H\"older continuous. Then we can define the \emph{Gurevi\v{c} pressure}
$\PG(\vp,\sigma) \in \mathbb R \cup\{\infty\}$ by
\[
\PG(\vp,\sigma) = \limsup_{n \to \infty} \frac{1}{n} \log \sum_{\sigma^nx=x} e^{\vp^n(x)} \mathbbm 1_{[a]}(x),
\]
for any choice of  $a \in S$ \cite{Sarig-etds}.
If $\sigma$ is topologically mixing, then the limsup may be replaced by a limit. 

The transfer operator $\L$ acts on real-valued function on $\Sigma$ by the formula
\[
(\L v)(x) = \sum_{\sigma y=x} e^{\vp(y)} v(y)
\]
(which may be infinite).
We say that $\vp$ is \emph{summable}
if $\|L_\vp 1\|_\infty<\infty$. 
For later use, we note that, using the local H\"older continuity of $\vp$, this is equivalent to the condition
\[
\sum_{a \in S} \exp\left(\sup\{\vp(y) \hbox{ : } y \in [a]\}\right)<\infty
\]
used in \cite{KK}, \cite{MU2001} and \cite{MU2003}.
If $\vp$ is summable then $\PG(\vp,\sigma)<\infty$.

We may also characterize the Gurevi\v{c} pressure directly in terms of the transfer operator.

\begin{lemma}[Theorem 4.4 of \cite{Sarig-survey}]\label{lem:GP_from_transfer}
If $\PG(\vp,\sigma)<\infty$ and $v : \Sigma \to \mathbb R$ is continuous, non-negative, not identically zero 
and with support contained inside a finite union of cylinders then
\[
\lim_{n \to \infty} \frac{1}{n} \log (\L^n v)(x) =\PG(\vp,\sigma)
\]
for all $x \in \Sigma$.
If the shift is topologically mixing then the limsup may be replaced by a limit.
\end{lemma}

\section{Skew product extension}\label{sec:skew_prod_ext}
We now consider skew product extensions.
Let $G$ be a finitely generated countable group (with the discrete topology), with identity $e$.
Suppose we are given a function $\psi : \Sigma \to G$ with $\psi((x_i)_{i=1}^\infty) = \psi(x_1)$.
(Any continuous function from $\Sigma$ to $G$ is locally constant and can be reduced to this form
by recoding.)
Then we can define the skew production extension
\[
T_\psi : \Sigma \times G \to \Sigma \times G
\]
by
\[
T_\psi(x,g) = (\sigma x,g\psi(x)).
\]
Note that
\[
T_\psi^n(x,g) = (\sigma^nx,g\psi_n(x)),
\]
where
\[
\psi_n(x) = \psi(x) \psi(\sigma x) \cdots \psi(\sigma^{n-1}x).
\]
We will always assume that $T_\psi$ is topologically transitive.

The system
$T_\psi : \Sigma \times G \to \Sigma \times G$ is also a countable state Markov shift with set of states
$S \times G$ and a transition matrix $\mathbb A$ defined by
\[
\mathbb A((i,g),(j,h))=1
\quad
\iff
\quad 
A(i,j)=1 \mbox{ and }  g\psi(i,j)=h.
\]
Thus, Gurevi\v{c} pressure is also defined for the extension.

Given an extension with group $G$, we will also need to consider the abelianized system
as follows.
Let $G^{\mathrm{ab}}$ denote the abelianization of $G$ and let $\overline G$ denote the torsion-free part of $G^{\mathrm{ab}}$, i.e. $\oG$ is the quotient of $G^{\mathrm{ab}}$ by its torsion subgroup.
Then $\oG$ is isomorphic to $\Z^d$, for some $d \ge 0$.
We use additive notation and let
$0$ to denote the identity.
Let $\pi :G \to \oG$ be the projection homomorphism.
Define $\bar \psi : \Sigma \to \oG$ by $\bar \psi = \pi \circ \psi$;
then we have the skew product $T_{\bar \psi} : \Sigma \times \oG \to \Sigma \times \oG$.
(If $d=0$ then we can regard $T_{\bar \psi}$ as $\sigma$.)

Given $\vp : \Sigma \to \R$, define $\tilde \vp : \Sigma \times G \to \R$ by
$\tilde \vp(x,g) = \vp(x)$ and 
$\bar \vp : \Sigma \times \overline{G} \to \R$ by 
$\bar \vp(x,k) = \vp(x)$. Following the convention in the introduction, 
we write all of these functions as $\vp$.
It is immediate from the definition of Gurevi\v{c} pressure that
\[
\PG(\vp,T_\psi) \le \PG(\vp,T_{\bar \psi}) \le \PG(\vp ,\sigma),
\]
Furthermore, if $T_{\bar \psi}^n(x,m)=(x,m)$ then $\sigma^nx=x$ and $\bar \psi^n(x)=0$, so we see that 
\[
\PG(\vp,T_{\bar \psi}) = \PG(\vp + \langle w,\bar \psi \rangle,T_{\bar \psi}),
\]
for all $w \in \R^d$. Together, these imply the inequalities (\ref{eq:gur_pres_inequalities}).

Each of the skew products has its own transfer operator (associated to $\vp$),
which we distinguish by using different fonts, as follows:
\[
(\mathcal L_{\vp}v)(x,g) := \sum_{T_\psi(y,h) = (x,g)} e^{\vp(y)} v(y,h)
= \sum_{\sigma y=x} e^{\vp(y)} v(y,g\psi(y)^{-1})
\]
and
\[
(\mathfrak L_{\vp}v)(x,m) := \sum_{T_{\bar \psi}(y,k) = (x,m)} e^{\vp(y)} v(y,k)
= \sum_{\sigma y=x} e^{\vp(y)} v(y,m-\bar \psi(y))
\]
The formulae for iterates are
\[
(\mathcal L_{\vp}^nv)(x,g) =
\sum_{T_\psi^n(y,h)=(x,g)} e^{\vp^n(y)} v(y,h)
=\sum_{\sigma^n y=x} e^{\vp^n(y)}
v(y,g\psi_n(y)^{-1})
\]
and
\[
(\mathfrak L_{\vp}^n v)(x,m) = \sum_{T_{\bar \psi}^n(z,k) = (x,m)} e^{\vp(y)} v(z,k)
= \sum_{\sigma y=x} e^{\vp(y)} v(y,m-\bar \psi_n(y)).
\]

\section{Amenability implies equality} \label{sec:am imp eq}

In the section we prove Theorem \ref{th:main}. Since we know from (\ref{eq:gur_pres_inequalities}) that
$\PG(\vp,T_\psi) \le \PG(\vp,T_{\bar \psi})$, we only
need to prove that $\PG(\vp,T_{\bar \psi}) \le \PG(\vp,T_\psi)$.

Fix $a \in S$; for $o \in [a]$ and $g \in G$, we will work with the quantity
\[
\sum_{\substack{\sigma^n x=o \\x_1=a\\ \psi_n(x)=g}} e^{\vp^n(x)}
=
\sum_{\substack{\sigma^n y=x}} e^{\vp^n(y)} \mathbbm 1_{[a] \times \{e\}}(x,g\psi_n(x)^{-1})
= (\mathcal L_{\vp}^n \mathbbm 1_{[a] \times \{e\}})(o,g).
\]

Since $T_\psi$ is topologically transitive, it follows from Lemma 
\ref{lem:GP_from_transfer} that
\begin{equation}\label{eq:GP_from_transfer_G}
\limsup_{n \to \infty} \frac{1}{n} \log (\mathcal L_{\vp}^n \mathbbm 1_{[a] \times \{e\}})(o,g)
= \PG(\vp,T_\psi).
\end{equation}
To simplify the notation, write $\rho = e^{\PG(\tilde \vp,T_\psi)}$. Define
\[
\eta_{o,g}(t) = \sum_{n=1}^\infty t^{-n} (\mathcal L_{\vp}^n\mathbbm 1_{[a] \times \{e\}})(o,g).
\]
This converges for $t>\rho$ but may either diverge or converge at $t=\rho$. In the latter case, we need to modify the series so that it diverges at $t=\rho$. We use the following standard lemma.

\begin{lemma}[Lemma 3.2 of \cite{DU}]\label{lem:DU}
Let $a_n$ be a sequence of positive real numbers such that $\limsup_{n \to \infty} a_n^{1/n} =\rho$.
Then there is a sequence $b_n$ of positive real numbers such that $\lim_{n \to \infty} b_{n+1}/b_n=1$,
$\sum_{n=1}^\infty a_nb_nt^{-n}$ converges for all $t > \rho$, but 
$\sum_{n=1}^\infty a_nb_n \rho^{-n}$ diverges.
\end{lemma}

Let $b_n$ be the sequence given by the lemma corresponding
to 
\[
a_n = (\mathcal L_{\vp}^n\mathbbm 1_{[a] \times \{e\}})(o,e)
\] 
and write
$c_n=1/b_n$. Let
\[
\zeta_{o,g}(t) 
= \sum_{n=1}^\infty t^{-n} \frac{1}{c_n} (\mathcal L_{\vp}^n\mathbbm 1_{[a] \times \{e\}})(o,g).
\]
If  $o' \in [a]$ then the is a natural bijection between the pre-images
of $o$ and $o'$, so we can compare $\zeta_{o,g}(t)$ and $\zeta_{o',g}(t)$ term-by-term. 
Using Lemma \ref{lem:bounded_variation}, for all $g \in G$ and $t >\rho$,
\begin{equation}\label{eq:bounds_for_zeta}
e^{-B}  \zeta_{o,g}(t) \le \zeta_{o',g}(t) \le e^B \zeta_{o,g}(t)
\end{equation}
and so the same sequence 
$c_n$ works for all  $o \in [a]$.

\begin{lemma}
Let $\rho'>\rho$.
For each $a \in S$ and each $g \in G$,
\[
0 < \inf_{o \in [a]} \inf_{\rho < t \le \rho'} \frac{\zeta_{o,g}(t)}{\zeta_{o,e}(t)} 
\le \sup_{o \in [a]} \sup_{\rho < t \le \rho'} \frac{\zeta_{o,g}(t)}{\zeta_{o,e}(t)} < \infty.
\]
\end{lemma}

\begin{proof}
We begin by observing that, for every $g,h \in G$, we have
\begin{align*}
(\mathcal L_{\vp}^{n+k}\mathbbm 1_{[a] \times \{e\}})(o,g)
&\ge
\sum_{\substack{\sigma^kz=o \\z_1=a\\ \psi_k(z)=h^{-1}g}} e^{\vp^k(z)}
\sum_{\substack{\sigma^nx=z \\x_1=a\\ \psi_n(x)= h}} e^{\vp^n(x)}
\\
&=
\sum_{\substack{\sigma^kz=o \\z_1=a \\ \psi_k(z)=h^{-1}g}} e^{\vp^k(z)}
[(\mathcal L_\vp^n\mathbbm 1_{[a] \times \{e\}})(z,h)]
\end{align*}
and (due to the transitivity of $T_\psi$) we can find a $k \ge 1$ so that the right
hand side is positive. Thus we obtain
\begin{align*}
\zeta_{o,g}(t) 
&= 
\sum_{m=1}^k \frac{t^{-n}}{c_n} (\mathcal L_\vp^m \mathbbm 1_{[a] \times
\{e\}})(o,g)
+ \sum_{n=1}^\infty \frac{t^{-(n+k)}}{c_{n+k}} 
(\mathcal L_{\vp}^{n+k}\mathbbm 1_{[a] \times \{e\}})(o,g)
\\
&\ge 
\sum_{m=1}^k \frac{t^{-n}}{c_n} (\mathcal L_\vp^m \mathbbm 1_{[a] \times
\{e\}})(o,g)
\\
&+
t^{-k}\sum_{\substack{\sigma^kz=o \\z_1=a\\ \psi_k(z)=h^{-1}g}} e^{\vp^k(z)}
\sum_{n=1}^\infty \frac{t^{-n}}{c_n} \frac{c_n}{c_{n+k}}
[(\mathcal L_\vp^n\mathbbm 1_{[a] \times \{e\}})(z,h)].
\end{align*}
Since the numbers $c_n$ are positive and, for each fixed $k$, 
$\lim_{n \to \infty} c_n/c_{n+k}=1$, we have
$\inf_{n \ge 1} c_n/c_{n+k} >0$.
Furthermore,
\[
\inf_{o \in [a]} \sum_{\substack{\sigma^kz=o \\z_1=a\\ \psi_k(z) = h^{-1}g}} e^{\vp^k(z)} >0.
\]
Hence, for $t \in (\rho,\rho']$, we have
\[
\zeta_{o,g}(t) \ge C_1(g,k,a) + \frac{C_2(g,h,k,a)}{C_3(k)}\zeta_{o,h}(t),
\]
for positive $C_1,C_2$ and $C_3$.
We conclude that
\[
\inf_{\rho <t\le \rho'} \frac{\zeta_{o,g}(t)}{\zeta_{o,h}(t)} >0.
\]
Since $g$ and $h$ are arbitrary, we have 
\[
0 <  \inf_{\rho < t \le \rho'} \frac{\zeta_{o,g}(t)}{\zeta_{o,e}(t)} 
\le  \sup_{\rho < t \le \rho'} \frac{\zeta_{o,g}(t)}{\zeta_{o,e}(t)} < \infty.
\]
Using (\ref{eq:bounds_for_zeta}), the bounds can be made uniform over $o \in [a]$.
\end{proof}

Define 
\[
\mathcal Z = \bigcup_{n=0}^\infty \{z\in \Sigma \hbox{ : } \sigma^nz=o\}
\]
the set of pre-images of $o$;
this is a 
countable subset of $\Sigma$.
By a standard diagonal argument, we can find a sequence $t_n \downarrow \rho$ such that
the following limit exists and lies in $(0,\infty)$ for all $(z,g) \in \mathcal Z \times G$:
\[
\Phi(z,g) := \lim_{n \to \infty} \frac{\zeta_{z,g}(t_n)}{\zeta_{z,e}(t_n)}.
\]
Furthermore, for all $g \in G$,
\begin{equation}\label{eq:bounds_for_Phi}
0 < \inf_{z\in \mathcal Z \cap [a]} \Phi(z,g) \le \sup_{z\in \mathcal Z \cap [a]} \Phi(z,g) <\infty.
\end{equation}

\begin{lemma}\label{lem:eigenvalue_equation}
For all $z \in \mathcal Z \cap [a]$ and all $n \ge 1$,
\[
(\mathcal L_{\vp}^n\Phi)(z,g) = \rho^n\Phi(z,g).
\]
\end{lemma}

\begin{proof}
To simplify notation, we shall write the proof for the case $z=o$ and note that, using
(\ref{eq:bounds_for_zeta}), the same argument holds for general $z \in \mathcal Z \cap [a]$.
Fix $k \ge 1$ and let $\epsilon>0$. Since
$\lim_{n \to \infty} c_{n-k}/c_n=1$, we can choose $n_0$ such that
$1-\epsilon \le c_{n-k}/c_h \le 1+\epsilon$, for all 
$n \ge n_0$.
We will also use that
\[
(\mathcal L_\vp^n \mathbbm 1_{[a] \times \{e\}})(o,g)
=
\sum_{\sigma^k z=o}
e^{\vp^k(z)} 
\,
(\mathcal L_\vp^{n-k} \mathbbm 1_{[a] \times \{e\}})(z,g\psi_k(z)^{-1}).
\]
Then, for $n \ge n_0$,
\begin{align*}
&\frac{1-\epsilon}{c_{n-k}}
\sum_{\sigma^k z=o}
e^{\vp^k(z)} 
\,
(\mathcal L_\vp^{n-k} \mathbbm 1_{[a] \times \{e\}})(z,g\psi_k(z)^{-1})
\le
\frac{1}{c_n} (\mathcal L_\vp^n \mathbbm 1_{[a] \times \{e\}})(o,g)
\\
&\le
\frac{1+\epsilon}{c_{n-k}}
\sum_{\sigma^k z=o}
e^{\vp^k(z)} 
\,
(\mathcal L_\vp^{n-k} \mathbbm 1_{[a] \times \{e\}})(z,g\psi_k(z)^{-1}).
\end{align*}
Setting
\[
C_1(g,t,n_0) = \sum_{n=1}^{n_0} 
\frac{t^{-n}}{c_n} (\mathcal L_\vp^n \mathbbm 1_{[a] \times \{e\}})(o,g),
\]
we have
\begin{align*}
&t^k \sum_{n=1}^\infty \frac{t^{-n}}{c_n} (\mathcal L_\vp^n \mathbbm 1_{[a] \times \{e\}})(o,g)
\le
C_1(g,t,n_0) 
\\
&+ t^k(1+\epsilon) \sum_{\sigma^kz=o}
e^{\vp^k(z)}
\,
\sum_{n=n_0+1}^\infty \frac{t^{-n}}{c_{n-k}}(\mathcal L_\vp^{n-k} \mathbbm 1_{[a] \times \{e\}})(z,g\psi_k(z)^{-1}) 
\\
&\le
C_1(g,t,n_0) + (1+\epsilon) \sum_{\sigma^kz=o} e^{\vp^k(z)}
\Phi(z,g\psi_k(z)^{-1}) \zeta_{o,e}(t).
\end{align*}
This gives
\begin{align*}
\rho^k \Phi(o,g) &\le \lim_{m\to \infty} \frac{C_1(g,t_m,n_0)}{\zeta_e(t_m)}
+(1+\epsilon) \sum_{\sigma^kz=o} e^{\vp^k(z)}
\Phi(z, g\psi_k(z)^{-1})
\\
&= (1+\epsilon) \sum_{\sigma^kz=o} e^{\vp^k(z)}
\Phi(z, g\psi_k(z)^{-1})
\end{align*}
and, since $\epsilon$ is arbitrary,
\begin{align*}
\rho^k \Phi(o,g) \le  \sum_{\sigma^kz=o} e^{\vp^k(z)}
\Phi(z, g\psi_k(z)^{-1})
\end{align*}

For the lower bound, we have 
\begin{align*}
&t^k \sum_{n=1}^\infty \frac{t^{-n}}{c_n}
(\mathcal L_\vp^n \mathbbm 1_{[a] \times \{e\}})(o,g)
\\
&\ge (1-\epsilon)\sum_{\sigma^k z=o} e^{\vp^k(z)}
\sum_{n=n_0+1}^\infty 
\frac{t^{-(n-k)}}{c_{n-k}}
(\mathcal L_\vp^{n-k} \mathbbm 1_{[a] \times \{e\}})(z,g\psi_k(z)^{-1})
\\
&=
(1-\epsilon) \sum_{\sigma^k z=o} e^{\vp^k(z)}
\sum_{n=n_0+k+1}^\infty 
\frac{t^{-n}}{c_{n}}
(\mathcal L_\vp^{n} \mathbbm 1_{[a] \times \{e\}})(z,g\psi_k(z)^{-1}),
\end{align*}
which gives
\[
\rho^k\Phi(o,g) \ge (\mathcal L_\vp^k \Phi)(o,g).
\]
\end{proof}

\begin{corollary}\label{cor:gamma_bounded}
For all $(z,\gamma) \in \mathcal Z \times G$, we have
\[
0 < \inf_{g \in G} \frac{\Phi(z, g\gamma^{-1})}{\Phi(o,g)} \le
\sup_{g \in G} \frac{\Phi(z,g\gamma^{-1})}{\Phi(o,g)} < 
\infty.
\]
\end{corollary}

\begin{proof}
Here we again use the topological transitivity of $T_\psi$.
Given $(z,\gamma) \in \mathcal Z \times G$, we can find $k \ge 1$ 
such that $\sigma^kz=o'\in [a]$ and $\psi_k(z)=\gamma$. This gives us
\begin{align*}
e^{\vp^k(z)}\Phi(z,g\gamma^{-1}) 
&\le
\sum_{\sigma^ky=o'} e^{\vp^k(y)} \Phi(y,g\psi_k(y)^{-1}) \\
&=\rho^k\Phi(o',g) \le e^{2B} \rho^k \Phi(o,g).
\end{align*}
Hence,
$\sup_{g \in G} \Phi(z,g\gamma^{-1})/\Phi(o,g)$ is finite.

Next we find $x \in [a]$ and $k \ge 1$ such that $\sigma^kx = z$ and $\psi_k(x)=\gamma^{-1}$.
Then
\begin{align*}
e^{-2B}\e^{\vp^k(x)}\Phi(o,g) 
&\le
e^{\vp^k(x)} \Phi(x,g) 
\\
&\le \sum_{\sigma^k y=o'} e^{\vp^k(y)} \Phi(y,g\gamma^{-1}\psi_k(y)^{-1})
= \rho^k\Phi(z,g\gamma^{-1}).
\end{align*}
Hence,
$\inf_{g \in G} \Phi(z,g\gamma^{-1})/\Phi(o,g)$ is positive.
\end{proof}

We are now able to use the amenability of $G$ to replace $\Phi$ with a function $\phi : \mathcal Z \times G \to \R^{>0}$ which, after normalization, is a homomorphism in the second factor. This will allow us to descend to $\oG$.

\begin{lemma}\label{lem:inequality_for_phi}
There is a function $\phi : \mathcal Z \times G \to \mathbb R^{>0}$ such that
\begin{enumerate}
\item
for all $g \in G$ and $n \ge 1$,
\[
(\mathcal L_{\vp}^n \phi)(o,g) \le \rho^n \phi(o,g),
\]
\item
for all $z \in \mathcal Z$, the function
\[
\gamma \mapsto \frac{\phi(z,\gamma)}{\phi(z,e)}
\]
is a homomorphism from $G$ to $\mathbb R^{>0}$.
\end{enumerate}
Furthermore, for all $a \in S$ and all $\gamma \in G$, we have
\[
\inf_{z \in \mathcal Z \cap [a]} \phi(z,\gamma) >0.
\]
\end{lemma}

\begin{proof}
The construction of the function $\phi$ uses the amenability of $G$.
Let $M$ be a right-invariant Banach mean on $\ell^\infty(G)$. By Jensen's inequality, if $\alpha$ is convex
and $f \in \ell^\infty(G)$
then
\[
M(\alpha(f)) \ge \alpha(M(f)).
\]
We apply this to the functions 
\[
f_{z,\gamma}(g) = \log \frac{\Phi(z,g\gamma^{-1})}{\Phi(o,g)},
\]
where $(z,\gamma) \in \mathcal Z \times G$.
It follows from Corollary \ref{cor:gamma_bounded} that $f_{z,\gamma} \in \ell^\infty(G)$, for all
$(z,\gamma) \in \mathcal Z \times G$. 
Thus, we have
\[
M\left(g \mapsto \frac{\Phi(z,g\gamma^{-1})}{\Phi(o,g)}\right)
= M\left(g \mapsto \exp \log \frac{\Phi(z,g\gamma^{-1})}{\Phi(o,g)}\right)
\ge \exp M(f_{z,\gamma}).
\]
Now define $\phi : \mathcal Z \times G \to \mathbb R^{>0}$ by
\[
\phi(z,\gamma) = \exp M(f_{z,\gamma}).
\]

We will first prove that the inequality in part (i) holds. We need a little care as $M$ is only finitely additive. Thus, we work with finite subsets of $G$ that exhaust $G$.
Let $\{g_k\}_{k=1}^\infty$ be an enumeration of $G$ and, for any $N \ge 1$, let 
$G_N = \{g_1,\ldots,g_N\}$. Then Lemma \ref{lem:eigenvalue_equation} gives us
\begin{align*}
\rho^n &=M\left(g \mapsto \frac{(\mathcal L_\vp^n \Phi)(o,g)}{\Phi(o,g)}\right)
= 
M
\left(g \mapsto \frac{1}{\Phi(o,g)}\sum_{\sigma^n z = o} e^{\vp^n(z)} \Phi(z,g\psi_n(z)^{-1})\right)
\\
&\ge
M
\left(g \mapsto \frac{1}{\Phi(o,g)}\sum_{\substack{\sigma^n z = o \\ \psi_n(z) \in G_N}} 
e^{\vp^n(z)} \Phi(z,g\psi_n(z)^{-1})\right)
\\
&=
\sum_{\substack{\sigma^n z = o \\ \psi_n(z) \in G_N}} e^{\vp^n(z)}
M\left(g \mapsto \frac{\Phi(z,g\psi_n(z)^{-1})}{\Phi(o,g)}\right)
\\
&\ge
\sum_{\substack{\sigma^n z = o \\ \psi_n(z) \in G_N}} e^{\vp^n(z)}
\exp M\left(g \mapsto \log \frac{\Phi(z,g\psi_n(z)^{-1})}{\Phi(o,g)}\right)
\\
&=
\sum_{\substack{\sigma^n z = o \\ \psi_n(z) \in G_N}} e^{\vp^n(z)}
\phi(z,g\psi_n(z)^{-1}).
\end{align*}
Taking the supremum over $N$ gives the inequality in (i).

Next we show that part (ii) holds.
First note that we have
\begin{align*}
\log \phi(o,\gamma \delta) &=
M\left(g \mapsto \log \frac{\Phi(o,g\delta^{-1}\gamma^{-1})}{\Phi(o,g)}\right)
\\
&= M\left(g \mapsto \log \frac{\Phi(o,g\delta^{-1}\gamma^{-1}}{\Phi(o,g\delta^{-1})}\right)
+
M\left(g \mapsto \log \frac{\Phi(o,g\delta^{-1})}{\Phi(o,g)}\right)
\\
&=
\phi(o,\gamma) + \phi(o,\delta)
\end{align*}
and
\begin{align*}
\log \phi(o,\gamma^{-1}) 
&= 
M\left(g \mapsto \log \frac{\Phi(o,g\gamma)}{\Phi(o,g)}\right)
=
M\left(g \mapsto \log \frac{\Phi(o,g)}{\Phi(o,g\gamma^{-1})}\right)
\\
&= 
M\left(g \mapsto -\log \frac{\Phi(o,g\gamma^{-1})}{\Phi(o,g)}\right)
= -\log \phi(o,\gamma).
\end{align*}
Hence $\gamma \mapsto \phi(o,\gamma)$ is a homomorphism.
Next, we observe that
\begin{align*}
\log \phi(z,\gamma) &= 
M\left(g \mapsto \log \frac{\Phi(z,g\gamma)}{\Phi(o,g)}\right)
\\
&= M\left(g \mapsto \log \frac{\Phi(o,g\gamma)}{\Phi(o,g)}\right)
+
M\left(g \mapsto \log \frac{\Phi(z,g\gamma)}{\Phi(o,g\gamma)}\right)
\\
&= M\left(g \mapsto \log \frac{\Phi(o,g\gamma)}{\Phi(o,g)}\right)
+
M\left(g \mapsto \log \frac{\Phi(z,g)}{\Phi(o,g)}\right)
\\
&= \log \phi(o,\gamma) + \log \phi(z,e).
\end{align*}
Hence
\[
\gamma \mapsto \frac{\phi(z,\gamma)}{\phi(z,e)} = \phi(o,\gamma)
\]
is a homomorphism.

To finish, we note that, by Corollary \ref{cor:gamma_bounded}, $\phi(z,\gamma)>0$ and 
one easily sees that this positive lower bound can be taken uniformly in $z \in \mathcal Z \cap [a]$.
\end{proof}

\begin{lemma}\label{lem:inequality_for_phi_bar}
There is a function $\bar \phi : \mathcal Z \times \oG \to \mathbb R^{>0}$
satisfying 
\[
(\mathfrak L_{\vp}^n \bar \phi)(o,g)\le \rho^n \bar \phi(o,g),
\]
for all $n \ge 1$, and,
for all $a \in S$ and all $m \in \oG$, we have
\[
\inf_{z \in \mathcal Z \cap [a]} \bar \phi(z,m) >0.
\]
\end{lemma}

\begin{proof}
By part (ii) of Lemma \ref{lem:inequality_for_phi}, for each $z \in \mathcal Z$,
there is a homomorphism
$\Upsilon_z : \oG \to \mathbb R^{>0}$ such that
\[
(\Upsilon_z \circ \pi)(\gamma) = \frac{\phi(z,\gamma)}{\phi(z,e)}.
\]
We deduce that if $\pi(\gamma)=\pi(\delta)$ then $\phi(z,\gamma)=\phi(z,\delta)$,
and we define $\bar\phi :\mathcal Z \times \oG \to \mathbb R^{>0}$ by
\[
\bar \phi(z,m) = \phi(z,\gamma)
\]
for any $\gamma \in G$ satisfying $\pi(\gamma)=m$.

We then have
\begin{align*}
(\mathfrak L_\vp^n\bar\phi)(o,0) &=
\sum_{\sigma^n z=o} e^{\vp^n(z)} \bar\phi(z,\bar{\psi}_n(z)^{-1})
\\
&=
\sum_{\sigma^n z=o} e^{\vp^n(z)} \phi(z,\psi_n(z)^{-1})
\\
&= (\mathcal L_\vp^n\phi)(o,e) \le \rho^n\phi(o,e) = \rho^n \bar\phi(o,0).
\end{align*}

The final statement follows from Lemma \ref{lem:inequality_for_phi} and the definition
of $\bar \phi$.
\end{proof}

We now complete the proof of Theorem \ref{th:main}.
Let 
\[
c = \inf_{z \in \mathcal Z \cap [a]} \bar\phi(z,0)>0.
\]
Using Lemma \ref{lem:GP_from_transfer}, we have that
\begin{align*}
\PG(\vp,T_{\bar \psi}) 
&=
\limsup_{n \to \infty}\frac{1}{n} \log (\mathfrak L_\vp^n (c\mathbbm 1_{[a] \times \{0\}}))(o,0)
\\
&\le 
\limsup_{n \to \infty}\frac{1}{n} \log (\mathfrak L_\vp^n \bar \phi)(o,0)
\le \rho,
\end{align*}
as required.

\begin{remark}
We could reformulate Theorem \ref{th:main} and the above proof in terms of the skew product
$T_{\psi^{\mathrm{ab}}} : \Sigma \times \Gab \to \Sigma \times \Gab$ without significantly changing the argument.
We observe that $\langle \xi,\cdot \rangle$ represents a choice of homomorphism from $\oG$ to $\R$ and 
there is a natural isomorphism between $\mathrm{Hom}(\oG,\R)$ and $\mathrm{Hom}(\Gab,\R)$.
Then $\langle \xi, \bar \psi \rangle$ is replaced by $\chi \circ \psi^{\mathrm{ab}}$, for an appropriate 
choice of $\chi \in \mathrm{Hom}(\Gab,\R)$.
\end{remark}

\section{The BIP property and skew product extensions}\label{sec:bip}

We now introduce the stronger condition on $\Sigma$ required for the equidistibution result we described
in the introduction.
Let $\sigma : \Sigma \to \Sigma$ be a countable state Markov shift.
We say that $\sigma : \Sigma \to \Sigma$ satisfies the \emph{big images and pre-images} (BIP) property if
there exist $b_1,\ldots,b_N \in S$ such that, for all $a \in S$, there exist
$i,j \in \{1,\ldots,N\}$ with $A(b_i,a)=1$ and $A(a,b_j)=1$.
Note that topologically transitive subshifts of finite type trivially satisfy this condition and, indeed,
BIP shifts share many thermodynamic properties with subshifts of finite type, in particular the existence of Gibbs measures.
(The terminology BIP appears in, for example, the work of Sarig. 
If $\sigma : \Sigma \to \Sigma$ is topologically mixing then BIP
is equivalent to the finite primitivity condition
used by Mauldin and Urba\'nski \cite{MU2001}, \cite{MU2003}. Thus, we can cite results from both sets of authors.)

Let $\vp : \Sigma \to \R$ be a locally H\"older continuous function.
We say that a $\sigma$-invariant probability measure $\mu$ on $\Sigma$ is a 
\emph{Gibbs measure} for $\vp$ if there are constants $A,B$ and $P$
(depending only on $\vp$) such that
\[
A \le \frac{\mu([w])}{e^{\vp^n(x)-nP}} \le B
\]
for all $w \in \W_n$, for all $n \ge 1$ and all $x \in [w]$. 
If $\vp$ has a Gibbs measure then $P=\PG(\vp,\sigma)$.

From now on, we assume that $\sigma : \Sigma \to \Sigma$ is topologically mixing and satisfies BIP.
Let $\vp : \Sigma \to \R$ be locally H\"older continuous and summable, so that 
$\PG(\vp,\sigma)$ is finite. Then $\vp$ has a unique Gibbs measure, which we denote by $\mu_\vp$, and $P=\PG(\vp)$ (Theorem 1
of \cite{Sarig-PAMS}) .
We also have a characterization of Gurevi\v{c} pressure in terms of periodic points:
\begin{equation}\label{eq:GP_periodic}
\PG(\vp,\sigma) = \lim_{n \to \infty} \frac{1}{n} \log \sum_{\sigma^nx=x} e^{\vp^n(x)}
\end{equation}
(Corollary 1 of \cite{Sarig-PAMS}). (Note that the equality (\ref{eq:GP_periodic}) may fail 
in the absence of the BIP property.)

In the presence of BIP, we also have stronger results for our transfer operators, acting
on an appropriate Banach space.
Let $\mathcal F_\theta^b$ denote the space of bounded $\theta$-locally H\"older continuous functions
$v: \Sigma \to \C$ with the norm $\|\cdot\|_\theta = \|\cdot\|_\infty +|\cdot|_\theta$,
where
\[
|v|_\theta=
\sup_{n \ge 1} \sup_{w \in \W_n} \sup_{x,y \in [w]} \frac{|v(x)-v(y)|}{\theta^n}.
\]
If $\vp : \Sigma \to \R$ is \emph{summable} then
$L_\vp : \mathcal F_\theta^b \to \mathcal F_\theta^b$ is a well-defined bounded
linear operator (Lemma 2.4.1 of \cite{MU2003}).

We have the following analyticity result.

\begin{lemma}[Theorem 2.6.13 and Proposition 2.6.13 of \cite{MU2003}]\label{lem:diff_of_GP}
Let $\sigma : \Sigma \to \Sigma$ be topologically mixing and satisfy BIP and let $\vp,\theta : \Sigma \to \R$ be weakly H\"older continuous. Suppose there exists $\epsilon>0$ such
that $\vp+t\theta$ is summable for $|t|<\epsilon$. Then $t \mapsto \PG(\vp+t\theta,\sigma)$ is real analytic on $(-\epsilon,\epsilon)$ and 
\[
\frac{d\PG(\vp+t\theta,\sigma)}{dt}
\Bigg|_{t=0}
=\int \theta\, d\mu_\vp.
\]
\end{lemma}

We need an analogue of equation (\ref{eq:GP_periodic}) for $T_\psi$.
If 
$T_\psi : \Sigma \times G \to \Sigma \times G$ is mixing then
we have the following lemma.

\begin{lemma}\label{lem:GP_BIP_sp}
Let $\sigma : \Sigma \to \Sigma$ be topologically mixing and satisfy BIP and, furthermore,
let $T_\psi : \Sigma \times G \to \Sigma \times G$ be topologically mixing.
Then we have
\[
\PG(\vp,T_\psi) = \lim_{n \to \infty} \frac{1}{n} \log \sum_{\substack{\sigma^nx=x \\\psi_n(x)=e}} e^{\vp^n(x)}.
\]
\end{lemma}

\begin{proof}
It is immediate from the definition of Gurevi\v{c} pressure that
\[
\PG(\vp,T_\psi) \le \lim_{n \to \infty} \frac{1}{n} \log \sum_{\substack{\sigma^nx=x \\\psi_n(x)=e}} e^{\vp^n(x)}.
\]
For the other direction, we adapt the proof of Theorem 1 in \cite{Sarig-PAMS}.
Given $w \in \mathcal W_n$, let
$
\vp^n_+(w) = \sup\{\vp^n(y) \hbox{ : } y \in [w]\}$.
Now fix $a \in S$. By the mixing of $T_\psi$,
there exists $n_0 \ge 1$ and, for $i=1,\ldots,N$, $w_{a,b_i}, w_{b_i,a} \in \mathcal W_{n_0}$
such that
$aw_{a,b_i}b_i \in \mathcal W_{n_0 +2}$, $b_iw_{b_i,a}a \in \mathcal W_{n_0+2}$,
$\psi_{n_0}(w_{a,b_i})=e$ and $\psi_{n_0}(w_{b_i,a})=e$.
By the BIP property, given $w \in \mathcal W_n$, there exists $i,j \in \{i,\ldots,N\}$ such that
$b_iwb_j 
\in \mathcal W_{n+2}$.
Hence
\[
aw_{a,b_i} b_i w b_j w_{b_j,a} a \in \mathcal W_{n+M+1},
\]
where $M = 2n_0+3$, and concatenating $aw_{a,b_i} b_i w b_j w_{b_j,a}$ gives a periodic point 
of period $n+M$.
Furthermore, if $\psi_n(w)=e$ then
\[
|\psi_{n+M}(aw_{a,b_i} b_i w b_j w_{b_j,a})|
\]
is bounded independently of $n$.
We conclude that there exists $n_1 \ge 1$ and $C >0$ so that
\begin{equation}\label{eq:GP_BIP_sp1}
\sum_{|g|\le n_1} \sum_{\substack{\sigma^{n+M}x=x \\ \psi_{n+M}(x)=g}} e^{\vp^{n+M}(x)}
\ge 
C
\sum_{\substack{w \in \mathcal W_n \\ \psi_n(w)=e}} e^{\vp^n_+(w)}
\ge 
C
(\mathcal L_\vp^n \mathbbm 1_{[a] \times \{e\}})(o,g),
\end{equation}
for any choice of $o \in [a]$.

Now, the mixing of $T_\psi$ implies that 
\[
\lim_{n \to \infty} \frac{1}{n} \log \sum_{\substack{\sigma^{n}x=x \\ \psi_{n}(x)=g}} e^{\vp^{n}(x)}
\]
is independent of $g \in G$. Combining this with (\ref{eq:GP_BIP_sp1}), we see that 
\[
\lim_{n \to \infty} \frac{1}{n} \log \sum_{\substack{\sigma^nx=x \\\psi_n(x)=e}} e^{\vp^n(x)}
\ge \PG(\vp,T_\psi),
\]
as required.
\end{proof}

We now wish to impose more conditions on $\psi$ and $\vp$.
 As above, let $\sigma : \Sigma \to \Sigma$ be mixing and satisfy BIP and let $\vp : \Sigma \to \R$ be a locally
 H\"older function which is summable.
Let $G$ be a finitely generated group and let $\psi : \Sigma \to G$ satisfy
$\psi(x)=\psi(x_1)$. 
To simplify notation, we identify $\bar \psi : \Sigma \to \oG$ with a function 
 $f : \Sigma \to \Z^d$. Let
 \[
 \delta(\vp,f) := \sup\{r \ge 0 \hbox{ : } L_{\vp+r|f|} \mbox{ is summable}\},
 \]
 where $|\cdot|$ denotes the $2$-norm on $\R^d$. We impose two assumptions on $\vp$
 and $\psi$.
 
  \medskip
\noindent
\emph{Assumption (I)}:
$T_\psi :\Sigma \times G \to \Sigma \times G$ is topologically mixing.

 \medskip
\noindent
\emph{Assumption (II)}:
 $\delta(\vp,f)>0$.
 
 \medskip
 Assumption (I) immediately implies that $T_f : \Sigma \times \Z^d \to \Sigma \times \Z^d$ is also topologically mixing. Furthermore, Assumption (I) implies that $f$ is \emph{aperiodic}:
 if $e^{2\pi i \langle t,f(x)\rangle} = z u(\sigma x)/u(x)$ for all $x \in \Sigma$, where $t \in \R^d$,
 $z \in \mathrm{U}(1) = \{z \in \C \hbox{ : }|z|=1\}$ and $u \in C(\Sigma,\mathrm{U}(1))$, then 
 $t=0\in \R^d/\Z^d$ and $z=1$.
 If Assumption (II) holds then we easily see that $\vp +\langle w,f \rangle$ is summable,
 $L_{\vp + \langle w,f \rangle} : \mathcal F_\theta^b \to \mathcal F_\theta^b$ is a well-defined bounded linear operator and 
 \[
 \mathfrak p(w):=\PG(\vp+\langle w,f \rangle,\sigma)
 \]
 is finite for $|w|<\delta(\vp,f)$.
 
 We also need to consider the transfer operator with complex potentials. Write $\delta=\delta(\vp,f)$ 
 and let 
 \[
 B(\delta)
 = \{w \in \R^d \hbox{ : } |w| < \delta\}
 \]
 and 
 \[
 B_{\C}(\delta) = \{s \in \C^d \hbox{ : } |\mathrm{Re}(s)|<\delta\}.
 \]
 One can easily check that $s \mapsto L_{\vp+\langle s,f\rangle}$ is analytic on
 $B_\C(\delta)$ (cf. Corollary 2.6.10 in \cite{MU2003}). 
 Furthermore, for $s$ in a neighbourhood of $B(\delta)$ in $B_\C(\delta)$,
 $L_{\vp +\langle s,f\rangle}$ has a simple isolated maximal eigenvalue $\lambda(s)$,
 such that 
 \[
\lambda(w) = e^{\mathfrak p(w)},
 \]
 for $w \in B(\delta)$ (Proposition 2.8 in \cite{KK}). As a consequence, 
 using spectral perturbation theory, we have that
 $s \mapsto e^{\mathfrak p(s)}$ is analytic for $s$ is a neighbourhood of $B(\delta)$ in $B_\C(\delta)$.

We may also calculate the derivatives of $\mathfrak p$. 
From Lemma \ref{lem:diff_of_GP}, we then have
\[
\nabla \mathfrak p(w)= \int f \, d\mu^{w},
\]
where 
\[
\mu^w := \mu_{\vp + \langle w.f \rangle}.
\]
Furthermore, the function $\mathfrak p$ is strictly convex. (Strict convexity fails only if unless there is $w \ne 0$ such that
$\langle w,f \rangle$ is cohomologous to a constant, which violates the transitivity of $T_f$.)

 \medskip

We make a further assumption on $\vp$ and $\psi$.

\medskip
 \noindent
 \emph{Assumption (III)}.
The function $w \mapsto \mathfrak p(w)$ has a unique minimum at $\xi \in \mathrm{int}(B(\delta))$.

\medskip

Under Assumption (III), we have
\[
\nabla \mathfrak p(\xi)= \int f \, d\mu^\xi =0.
\]

\begin{remark}
Assumptions (II) and (III) are similar in spirit to the ``entropy gap'' conditions that appear in the recent
papers 
\cite{BCKM} and \cite{KM}.
\end{remark}

The value $\xi$ is important and allows us to relate the Gurevi\v{c} pressure $\PG(\vp,T_f)$
to a Gurevi\v{c} pressure with respect to $\sigma$. First, we give a technical lemma.
We observe that since $f$ is valued in $\Z^d$, $L_{\vp+\langle \xi+2\pi it,f\rangle}$ is well-defined 
for $t \in \R^d/\Z^d$.

\begin{lemma}\label{lem:smaller_spectral_radius}
For $t \in \R^d/\Z^d \setminus \{0\}$,
$L_{\vp + \langle \xi +2\pi it,f \rangle}$
has spectral radius strictly less than $e^{\mathfrak p(\xi)}$.
\end{lemma}

\begin{proof}
As noted above Assumption (I) implies that $f$ is aperiodic. It then follows from Theorem 2.14 of
\cite{KK} that $L_{\vp + \langle \xi +2\pi it,f \rangle}$
has spectral radius strictly less than $e^{\mathfrak p(\xi)}$.
\end{proof}

\begin{proposition}\label{prop:magic_xi}
We have
\[
\PG(\vp,T_f) = \PG(\vp+\langle \xi,f\rangle,\sigma).
\]
\end{proposition}

\begin{proof}
By the definition of Gurevi\v{c} pressure (and recalling our convention that $\vp : \Sigma \times \Z^d \to \R$
is defined by $\vp(x,\cdot)=\vp(x)$), we have
\[
\PG(\vp,T_f)
= \lim_{n \to \infty} \frac{1}{n} \log \sum_{\substack{\sigma^nx=x \\ f^n(x)=0}}
e^{\vp^n(x)} \mathbbm{1}_{[a]}(x).
\]
On the other hand, we may make the following calculation. Choose $o \in [a]$ and set
\[
\mathcal Z(n,a,o) = \{y \in [a] \hbox{ : } \sigma^ny=o\}.
\]
There is a natural bijection $q : \mathrm{Fix}_n(\sigma) \cap [a] \to \mathcal Z(n.a,o)$ given by
$q(x) = x_1\ldots x_no$.
We have
\begin{align*}
\sum_{\substack{\sigma^nx=x \\ f^n(x)=0x}} e^{\vp^n(x)} \mathbbm{1}_{[a]}(x)
&=\sum_{\substack{\sigma^nx=x \\ f^n(x)=0}} e^{\vp^n(x)+\langle \xi,f^n(x)\rangle} \mathbbm{1}_{[a]}(x)
\\
&= \int_{\R^d/\Z^d} \sum_{x \in \mathrm{Fix}_n(\sigma) \cap [a]}
e^{\vp^n(x)+\langle \xi+2\pi it,f^n(x)\rangle} \, dt
\\
&= \int_{\R^d/\Z^d} \sum_{y \in \mathcal Z(n,a,z)} e^{\vp^n(y)+\langle \xi+2\pi it,f^n(y)\rangle} \, dt
+ O(\theta^n e^{n\mathfrak p(\xi)})
\\
&=
\int_{\R^d/\Z^d} (L^n_{\vp+\langle \xi+2\pi i t,f\rangle} \mathbbm{1}_{[a]})(z)
+ O(\theta^n e^{n\mathfrak p(\xi)}),
\end{align*}
for some $0<\theta<1$. Using Lemma \ref{lem:smaller_spectral_radius} and the analysis in \cite{PS94}, we can show that
the right hand side behaves like $e^{n\mathfrak p(\xi)} n^{-d/2}$, so the result follows.
\end{proof}

\section{Equidistribution for skew product extensions}\label{sec:equi_for_skew}
In this section, we obtain a weighted equidistribution result for periodic points and pre-images with respect to
skew product extensions  of countable state Markov shifts with the BIP property satisfying Assumptions (I), (II) and (III).

For $x \in \Sigma$, let $\delta_x$ denote the Dirac measure at $x$ and write
\[
\tau_{x,n} = \frac{1}{n} \sum_{j=0}^{n-1} \delta_{\sigma^jx}.
\]
For a sequence of countable sets $\Lambda_n \subset \Sigma$, we write
\[
\Pi_\vp(\Lambda_n) = \sum_{x \in \Lambda_n} e^{\vp^n(x)}
\]
and introduce a sequence of probability measures on $\Sigma$,
\[
\mathfrak M_\vp(\Lambda_n) =  \frac{1}{\Pi_\vp(\Lambda_n)} \sum_{x \in \Lambda_n} e^{\vp^n(x)} \tau_{x,n}.
\]
We consider the following choices for $\Lambda_n$:
\begin{enumerate}
\item[(1)]
\[
\Lambda^\psi(n) =\{x \in \Sigma \hbox{ : } \sigma^n x=x, \ \psi_n(x) =0\};
\]
\item[(2)]
for $a \in S$,
\[
\Lambda_a^\psi(n) =\{x \in [a] \hbox{ : } \sigma^n x=x, \ \psi_n(x) =0\};
\]
\item[(3)]
for $o \in \Sigma$,
\[
\Lambda^\psi(n,o) = \{x\in \Sigma \hbox{ : } \sigma^nx=o, \ \psi_n(x)=e\};
\]
\item [(4)]
for $a \in S$ and $o \in [a]$,
\[
 \Lambda_a^\psi(n,o) = \{x\in [a] \hbox{ : } \sigma^nx=o, \ \psi_n(x)=e\}.
 \]
\end{enumerate}

\begin{theorem}\label{th:equi_for_G}
Let $\sigma : \Sigma \to \Sigma$ be a countable state Markov shift which satisfies BIP
and let $G$ be a finitely generated group. Let $\vp : \Sigma \to \R$ be locally H\"older continuous 
and let $\psi : \Sigma \to G$
satisfy $\psi(x)=\psi(x_1)$, such that $\vp$ and $\psi$ satisfy Assumptions (I), (II) and (III).
If $G$ is amenable then
$\mathfrak M_\vp(\Lambda^\psi(n))$,
$\mathfrak M_\vp(\Lambda_a^\psi(n))$,
$\mathfrak M_\vp(\Lambda^\psi(n,o))$ and
$\mathfrak M_\vp(\Lambda_a^\psi(n,o))$
all converge to $\mu^\xi$, as $n \to \infty$,
with respect to the weak$^*$ topology on $\mathcal M(\Sigma)$.
\end{theorem}

We give the proof for the sequence $\mathfrak M_\vp(\Lambda^\psi(n))$,
the other cases being similar. The proof is based on the following lemma.

\begin{lemma}\label{lem:ld_estimate}
Let $g : \Sigma \to \R$ be a bounded locally H\"older continuous function.
Given $\epsilon>0$, there exists $C>0$ and $\eta>0$ such that
\[
\frac{1}{\Pi_\vp(\Lambda^\psi(n))} \sum_{\substack{x \in \Lambda^\psi(n) \\
|\int g \, d\tau_{x,n} -\int g \, d\mu_\xi|>\epsilon}}
e^{\vp^n(x)} \le Ce^{-\eta n}.
\]
\end{lemma}

\begin{proof}
We consider $x \in \Lambda^\psi(n)$ such that 
$\int g \, d\tau_{x,n} > \int g \, d\mu_\xi +\epsilon$ and
$\int g \, d\tau_{x,n} < \int g \, d\mu_\xi -\epsilon$ separately.
For $s>0$, we have
\[
\sum_{\substack{x \in \Lambda^\psi(n) \\ \int g \, d\tau_{x,n} >\int g \, d\mu^\xi +\epsilon}}
e^{\vp^n(x)}
\le 
\sum_{x \in \Lambda^\psi(n)} e^{\vp^n(x) + sg^n(x) - ns\int g \, d\mu^\xi -ns\epsilon},
\]
so that
\[
\limsup_{n \to \infty} \frac{1}{n} \log 
\sum_{\substack{x \in \Lambda^\psi(n) \\ \int g \, d\tau_{x,n} >\int g \, d\mu^xi +\epsilon}}
e^{\vp^n(x)}
\le \PG(\vp+\langle \xi,f\rangle +sg,\sigma) -s\int g \, d\mu^\xi -s\epsilon.
\]
Similarly, for $s<0$, we have
\[
\limsup_{n \to \infty} \frac{1}{n} \log 
\sum_{\substack{x \in \Lambda^\psi(n) \\ \int g \, d\tau_{x,n} <\int g \, d\mu^\xi -\epsilon}}
e^{\vp^n(x)}
\le \PG(\vp+\langle \xi,f\rangle +sg,\sigma) -s\int g \, d\mu^\xi +s\epsilon.
\]
If we write $\mathfrak p_g(s) := \PG(\vp+\langle \xi,f\rangle +sg)$ then,
applying Lemma \ref{lem:diff_of_GP}, we have a bound of the form
\[
\limsup_{n \to \infty} \frac{1}{n} \log
\sum_{\substack{x \in \Lambda^\psi(n) \\
|\int g \, d\tau_{x,n} -\int g \, d\mu_\xi|>\epsilon}}
e^{\vp^n(x)} 
\le
\mathfrak p_g(s) - s\mathfrak p_g'(0) -\epsilon |s|,
\]
which is strictly smaller than $\PG(\vp+\langle \xi,f \rangle)$ for sufficiently small values of $s$. 
The proof is completed by noting that, by
Theorem \ref{th:main} and
Proposition \ref{prop:magic_xi},
\[
\lim_{n \to \infty} \frac{1}{n} \log \Pi_\vp(\Lambda^\psi(n)) = \PG(\vp+\langle \xi,f \rangle,\sigma).
\]
\end{proof}

\begin{proof}[Proof of Theorem \ref{th:equi_for_G}]
Write $\mathfrak M_n = \mathfrak M_\vp(\Lambda^\psi(n))$.
We need to show that for every bounded continuous function
$g : \Sigma \to \R$ we have
\[
\lim_{n \to \infty} \int g \, d\mathfrak M_n = \int g \, d\mu^\xi.
\]
It is enough to prove this when $g$ is a bounded locally H\"older continuous function.
We have
\begin{align*}
\int g \, d\mathfrak M_n - \int g \, d\mu^\xi 
&= \frac{1}{\Pi_\vp(\Lambda^\psi(n))}
\sum_{\substack{x \in \Lambda^\psi(n) \\ \left|\int \chi \, d\tau_{x,n} - \int \chi d\mu^\xi\right| > \epsilon}} 
e^{\vp^n(x)}
\int g \, d\tau_{x,n} 
\\
&+  \frac{1}{\Pi_\vp(\Lambda^\psi(n))}\sum_{\substack{x \in \Lambda^\psi(n) \\ \left|\int \chi \, d\tau_{x,n} - \int g d\mu^\xi\right| \le \epsilon}}  
e^{\vp^n(x)}
\int g \, d\tau_{x,n}-\int \chi \, d\mu^\xi.
\end{align*}
By Lemma \ref{lem:ld_estimate}, the first term on the right hand side tends to zero
exponentially fast. Also,
\begin{align*}
&\left|
\frac{1}{\Pi_\vp(\Lambda^\psi(n))}
\sum_{\substack{x \in \Lambda^\psi(n) \\ \left|\int g \, d\tau_{x,n} - \int g d\mu^\xi\right| \le \epsilon}} 
e^{\vp^n(x)}
\int g \, d\tau_{x,n}-\int g \, d\mu^\xi
\right|
\\
&\le
\epsilon + 
\frac{\left|\int g \, d\mu^\xi\right|}{\Pi_\vp(\Lambda^\psi(n))} 
\sum_{\substack{x \in \Lambda^\psi(n) \\ \left|\int g \, d\tau_{x,n} - \int g d\mu^\xi\right| > \epsilon}} 
e^{\vp^n(x)}
\le \epsilon + Ce^{-\eta(\epsilon)n},
\end{align*}
which, since $\epsilon$ is arbitrary, gives the result.

\end{proof}

\section{Equality implies amenability}\label{sec:eq imp am}

In this section we show that, under Assumptions (I), (II) and (III), the converse of Theorem 
\ref{th:main}
holds, i.e.
if $\PG(\vp,T_\psi)=\PG(\vp,T_{\bar \psi})$ then $G$ is 
amenable. As above, we write $f = \bar{\psi}$.

To do this, we need to specify a function space for $\mathcal L_\vp$ to act on.
Let $\mathcal H$ denote the set of continuous functions
$v : \Sigma \times G \to \C$ such that $g \mapsto \|v(\cdot,g)\|_\infty$ is in $\ell^2(G)$, with norm
\[
\|v\|_{\mathcal H} = \left(\sum_{g \in G} \|v(\cdot,g)\|_\infty^2\right)^{1/2},
\]
and let $\mathrm{spr}_{\mathcal H}(\mathcal L_\vp)$ denote the spectral radius of $\mathcal L_\vp$ on
$\mathcal H$.
We always have 
\[
\PG(\vp,T_\psi) \le \log \mathrm{spr}_{\mathcal H}(\mathcal L_{\vp}) \le \PG(\vp,\sigma).
\]

We state a result due to Stadlbauer.

\begin{proposition}[Stadlbauer \cite{Stadlbauer2013}, Theorem 5.4]\label{prop:Stadlbauer}
Suppose that $\sigma : \Sigma \to \Sigma$  satisfies BIP 
and that $T_\psi : \Sigma \times G \to \Sigma \times G$ is topologically transitive.
If $G$ is not amenable then $\log \mathrm{spr}_{\mathcal H}(\mathcal L_\vp) <\PG(\vp,\sigma)$.
\end{proposition}

We now have the desired result.

\begin{proposition}\label{prop:eia}
If $\PG(\vp,T_\psi)=\PG(\vp,T_{\bar \psi})$ then 
$G$ is amenable.
\end{proposition}

\begin{proof}
Let $\xi \in \R^d$ be as in the previous section and let $\vp_\xi = \vp +\langle \xi,f \rangle$.
Suppose that $G$ is not amenable. Then
\[
\PG(\vp,T_\psi) = \PG(\vp_\xi,T_\psi) \le \log \mathrm{spr}_{\mathcal H}(\mathcal L_{\vp_\xi})
< \PG(\vp_\xi,\sigma) = \PG(\vp,T_{\bar \psi}),
\]
where the strict inequality uses Proposition \ref{prop:Stadlbauer}.
\end{proof}


\begin{thebibliography}{90}

\bibitem{BabLed}
M.~Babillot and F.~Ledrappier,
\emph{Lalley’s theorem on periodic orbits of hyperbolic flows}, 
Ergodic Theory Dynam. Sys. 18, 17--39, 1998.

\bibitem{Bowen1971}
R.~Bowen,
\emph{Periodic points and measures for Axiom A diffeomorphisms},
Trans. Amer. Math. Soc. 154, 377--397, 1971.

\bibitem{BCKM}
H.~Bray, R.~Canary, L.-Y.~Kao and G. Martone,
\emph{Counting, equidistribution and entropy gaps at infinity with applications to
cusped Hitchin representations},
J. Reine Angew. Math. 791, 1--51, 2022.

\bibitem{Brooks1981}
R.~Brooks,
\emph{The fundamental group and the spectrum of the Laplacian},
Comment. Math. Helv. 56, 581--598, 1981.

\bibitem{Brooks1985}
R.~Brooks,
\emph{The bottom of the spectrum of a Riemannian covering},
J. Reine Angew. Math. 357, 101--114, 1985.

\bibitem{Cohen1982}
J.~Cohen,
\emph{Cogrowth and amenability in discrete groups},
J. Funct. Anal. 48, 301--309, 1982.

\bibitem{ColesSharp}
S.~Coles and R.~Sharp,
\emph{Helicity, linking and the distribution of null-homologous periodic orbits for Anosov flows},
 Nonlinearity 36, 21--58, 2023.

\bibitem{CDS}
R.~Coulon, F.~Dal'Bo and A.~Sambusetti,
\emph{Growth gap in hyperbolic groups and amenability},
Geom. Funct. Anal. 28, 1260--1320, 2018.

\bibitem{CDST}
R.~Coulon, R.~Dougall, B.~Schapira and S.~Tapie,
\emph{Twisted Patterson--Sullivan measures and applications to amenability and coverings},
Mem. Amer. Math. Soc. 

\bibitem{DU}
M.~Denker and M.~Urba\'nski,
\emph{On the existence of conformal measures},
Trans. Amer. Math. Soc. 328, 563--587, 1991.

\bibitem{Dougall2019}
R.~Dougall,
\emph{Critical exponents of normal subgroups, the spectrum of group extended transfer operators, and Kazhdan distance},
Adv. Math. 349, 316--347, 2019.

\bibitem{DougallSharp2016}
R.~Dougall and R.~Sharp,
\emph{Amenability, critical exponents of subgroups and growth of closed geodesics},
Math. Ann. 365, 1359--1377, 2016.

\bibitem{DougallSharp2021}
R.~Dougall and R.~Sharp,
\emph{Anosov flows, growth rates on covers and group extensions of subshifts},
Invent. math. 223, 445--483, 2021.
(Correction: Invent. math. 236, 1505--1509, 2024.)

\bibitem{DougallSharp2024}
R.~Dougall and R.~Sharp,
\emph{A non-symmetric Kesten criterion and ratio limit theorem for random walks on amenable groups},
Int. Math. Res. Notices Volume 2024, Issue 7, 6209--6223, 2024.

\bibitem{GomezTerhesiu2025}
J.~Gomez and D.~Terhesiu,
\emph{Ratio limits and pressure function for group extensions of Gibbs Markov maps},
arXiv:2507.08186.

\bibitem{Grigorchuk1980}
R.~Grigorchuk,
\emph{Symmetrical random walks on discrete groups},
Multicomponent Random Systems, Adv. Probab. RelatedTopics 6, Dekker, New York, 1980, pp. 285--325.



\bibitem{KM}
L.-Y.~Kao and G.~Martone,
\emph{Correlation number for potentials with entropy gaps and cusped Hitchin representations},
arXiv:2412.20627.

\bibitem{KK}
M.~Kesseb\"ohmer and S.~Kombrink,
\emph{A complex Ruelle--Perron--Frobenius theorem for infinite Markov shifts with applications
to renewal theory},
Disc. Cont. Dyn. Sys. 10, 335--352, 2017.

\bibitem{Kesten}
H.~Kesten,
\emph{Full Banach mean values on countable groups},
Math. Scand. 7, 146--156, 1959. 




\bibitem{LalleyAdv}
S.~Lalley,
\emph{Distribution of periodic orbits of symbolic and Axiom A flows},
Adv. Appl. Math. 8, 154--193, 1987.


\bibitem{MU2001}
R.~D.~Mauldin and M.~Urba\'nski,
\emph{Gibbs states on the symbolic spaces over an infinite alphabet},
Israel J. Math. 125, 93--130, 2001.

\bibitem{MU2003}
R.~D.~Mauldin and M.~Urba\'nski,
\emph{Graph Directed Markov Systems: Geometry and Dynamics of Limit Sets},
Cambridge Tracts in Mathematics 148, Cambridge University Press, Cambridge, 2003.

\bibitem{Northshield1992}
S.~Northshield,
\emph{Cogrowth of regular graphs},
Proc. Amer. Math. Soc. 116, 203--205, 1992.

\bibitem{Northshield2004}
S.~Northshield,
\emph{Cogrowth of arbitrary graphs},
Random Walks and Geometry, de Gruyter, Berlin 2004, pp. 501--513.

\bibitem{Parry1988}
W.~Parry,
\emph{Equilibrium states and weighted uniform distribution of closed orbits}, Lecture Notes in Math., 1342, Springer--Verlag, Berlin, 617--625, 1988.

\bibitem{PP}
W.~Parry and M.~Pollicott,
\emph{Zeta functions and the periodic orbit functions of hyperbolic dynamics},
Ast\'erisque 187--188, Soci\'et\'e Math\'ematique de France, 1990.



\bibitem{PS94}
M.~Pollicott and R.~Sharp,
\emph{Rates of recurrence for $\Z^q$ and $\R^q$ extensions of subshifts of finite type},
J. London Math. Soc. 49, 401--416, 1994.

\bibitem{Roblin}
T.~Roblin,
\emph{Un th\'eor\`eme de Fatou pour les densit\'es conformes avec applications
aux rev\^etements Galoisiens en courbure n\'egative},
Israel J. Math. 147, 333--357, 2005.


\bibitem{Sarig-etds}
O.~Sarig,
\emph{Thermodynamic formalism for countable Markov shifts},
Ergodic Theory Dynam. Systems 19, 1565--1593, 1999.


\bibitem{Sarig-PAMS}
O.~Sarig,
\emph{Existence of Gibbs measures for countable state Markov shifts},
Proc. Amer. Math. Soc. 131, 1751--1758, 2003.

\bibitem{Sarig-survey}
O.~Sarig,
\emph{Thermodynamic formalism for countable Markov shifts} in Hyperbolic dynamics, fluctuations
and large deviations, 81--117,
Proc. Sympos. Pure Math. 89, American Mathematical Society, Providence, RI, 2015.


\bibitem{Sharp93}
R.~Sharp,
\emph{Closed orbits in homology classes for Anosov flows},
Ergodic Theory Dynam. Sys. 13, 387--408, 1993.


\bibitem{Stadlbauer2013}
M.~Stadlbauer,
\emph{An extension of Kesten's criterion for amenability to topological Markov chains},
Adv. Math. 235, 450--468, 2013.


\bibitem{Stadlbauer2019}
M.~Stadlbauer,
\emph{On conformal measures and harmonic functions for group extensions}.
In the Proceedings of ``New Trends in One-Dimensional Dynamics'', ed. M.~J.~Pacifico
and P.~Guarino, Springer Proceedings in Mathematics \& Statistics, pages 275--304,
Springer Nature Switzerland, 2019.


\bibitem{Takahasi2019}
H.~Takahasi,
\emph{Large deviation principles for countable Markov shifts},
Trans. Amer. Math. Soc. 372, 7831--7855, 2019.



\end{thebibliography}
\end{document}